\documentclass[a4paper,11pt]{article}
\usepackage{geometry,latexsym,amssymb,amsmath,color}
\usepackage[latin5]{inputenc}
\usepackage{enumerate}
\usepackage{enumitem}
\usepackage{hyperref}
\usepackage[all]{xy}
\usepackage{authblk}
\usepackage{tikz}
\usepackage{tikz-cd}
\usepackage{amsthm}
\usepackage{amsfonts}

\usetikzlibrary{matrix,arrows}

\newtheorem{theorem}{Theorem}[section]

\newtheorem{corollary}[theorem]{Corollary}

\newtheorem{definition}[theorem]{Definition}

\newtheorem{proposition}[theorem]{Proposition}

\def\Mon{\operatorname{Mon}}
\def\M(G,W){\operatorname{M(G,W)}}
\def\F(W){\operatorname{F(W)}}
\def\Hol{\operatorname{Hol}}
\def\Ob{\operatorname{Ob}}
\def\TGrpGpd{\mathsf{STopGrpGpd}}
\def\STopGrpGpd{\mathsf{StarTopGrpGpd}}
\def\TopGrpGpd{\mathsf{TopGrpGpd}}

\begin{document}
\title{\large\bf Topological aspect of  monodromy groupoid for  a group-groupoid}
\author[1]{Osman MUCUK \thanks{mucuk@erciyes.edu.tr}}
\author[2]{Serap DEM\.{I}R \thanks{srpdmr89@gmail.com}}
\affil[1]{Department of Mathematics, Erciyes University Kayseri 38039, Turkey}
\affil[2]{Department of Mathematics, Erciyes University Kayseri 38039, Turkey}
\date{\vspace{-5ex}}
\maketitle

\\ {\bf Classification:} Primary  20L05, 57M10, ; Secondary  22AXX, 22A22



\begin{abstract}
In this paper we develop  star topological and  topological group-groupoid structures  of monodromy groupoid and prove that the monodromy groupoid of a topological group-groupoid is also a topological group-groupoid.
\end{abstract}


\section{Introduction}

As enunciated by Chevalley in \cite[Theorem 2, Chapter 2]{Ch}, the general idea of the  monodromy principle  is that of extending a local morphism $f$  on a topological structure $G$, or extending  a restriction of $f$, not to $G$ itself but to some simply connected cover of $G$. A form of this for topological groups was given in \cite[Theorem 3]{Ch}, and developed by Douady and Lazard in \cite{Do-La} for Lie groups, generalized to topological groupoid case in \cite{Mu1} and \cite{Br-Mu2}.

The notion of monodromy groupoid was indicated by J. Pradines in \cite{Pr} as part of his grand scheme announced in \cite{Pr,Pr2,Pr3,Pr4} to generalize the standard construction of a simply connected Lie group from a Lie algebra to a corresponding construction of a Lie groupoid from a Lie algebroid (see also \cite{Ma,Ma2,Ph}).

 One  construction of the monodromy groupoid for a topological groupoid $G$ and an open subset $W$ including the identities is given via free groupoid concept and denoted by $\M(G,W)$ as a generalization of the construction  in \cite{Do-La}.

Another construction of the   monodromy groupoid  for a topological groupoid $G$ in which each star $G_x$ has a universal cover is directly given in  Mackenzie~\cite[p.67-70]{Ma} as a disjoint union of the universal covers  of the stars $G_x$'s and denoted by $\Mon(G)$.

 These  two monodromy groupoids $\M(G,W)$ and $\Mon(G)$ are identified as star Lie  groupoids in \cite[Theorem 4.2]{Br-Mu2} using Theorem \ref{TheoremHolonomy}  which is originally \cite[Theorem 2.1]{Ao-Br} to get an appropriate topology.

In particularly if $G$ is a connected topological group which has a universal cover, then the monodromy groupoid $\Mon( G)$   is the universal covering group, while if  $G$ is the topological groupoid  $X\times X$,  for   a semi-locally simply connected topological space $X$, then the monodromy groupoid $\Mon(G)$ is the fundamental groupoid  $\pi X $. Hence the monodromy groupoid generalizes both the concepts of universal covering group and the fundamental groupoid. For further discussion on monodromy and holonomy groupoids see \cite{BIM}.

The notion of monodromy groupoid for  topological  group-groupoid was recently introduced and investigated in  \cite{Mu-Be-Tu-Na}; and then it has been generalized to the internal groupoid case in \cite{Mu-Ak}. Motivated by the referee's comments of latter paper, in this paper, we aim to develop the topological aspect of  the  monodromy groupoid $\Mon(G)$ as a group-groupoid for a topological group-groupoid $G$.

The organization of the paper is as follows: In Section 1 we give a preliminary of groupoid, group-groupoid,  topological groupoid and the constructions of  monodromy groupoid. In section 2  star topological group-groupoid structure of monodromy groupoid with some results are given.   Section 3 is devoted to develop topological aspect of the  monodromy groupoid $\Mon(G)$ as group-groupoid together with a strong monodromy principle for topological group-groupoids.

\section{Preliminaries on monodromy groupoid}
A {\em groupoid}  is a small category in which each morphism is an isomorphism (see for example \cite{Br1} and \cite{Ma}). So a groupoid $G$ has a set $G$ of morphisms and  a set $\Ob(G)$ of {\em objects} together with {\em source} and {\em target} point maps $\alpha, \beta\colon G\rightarrow \Ob(G)$  and {\em object inclusion} map  $\epsilon \colon \Ob(G) \rightarrow G$  such that $\alpha\epsilon=\beta\epsilon=1_{\Ob(G)}$. There exists a partial composition defined by $ G_\beta\times_\alpha G\rightarrow G, (g,h)\mapsto g\circ h$, where $G_\beta\times_\alpha G$ is the pullback of $\beta$ and $\alpha$.  Here if $g,h\in G$ and $\beta(g)=\alpha(h)$, then the {\em composite}  $g\circ h$ exists such that $\alpha(g\circ h)=\alpha(g)$ and $\beta(g\circ h)=\beta(h)$. Further, this partial composition is associative, for $x\in \Ob(G)$ the morphism $\epsilon (x)$ acts as the identity and it is denoted by $1_x$, and each element $g$ has an inverse $g^{-1}$ such that $\alpha(g^{-1})=\beta(g)$,  $\beta(g^{-1})=\alpha(g)$,   $g\circ g^{-1}=\epsilon (\alpha(g))$, $g^{-1}\circ g=\epsilon (\beta((g))$. The map $G\rightarrow G$, $g\mapsto g^{-1}$ is called the {\em inversion}. In  a groupoid $G$, the source and target points, the object inclusion, the inversion maps and the partial composition  are  called {\em structural maps}. An example of a groupoid is  fundamental groupoid of a topological space $X$, where the objects are points of $X$ and  morphisms are homotopy classes of the paths relative to the end points. A group is a groupoid with one object.

In a groupoid $G$  each object $x$ is identified with unique identity $\epsilon(x)=1_x$ and hence we sometimes  write $\Ob(G)$ for the set of identities.  For $x,y\in \Ob(G)$ we write $G(x,y)$ for  $\alpha^{-1}(x)\cap \beta ^{-1}(y)$. The { difference  map } $\delta \colon G\times_{\alpha} G\rightarrow G$ is given by  $\delta (g,h) ={g}^{-1} \circ h$, and is defined on the double pullback of  $G$ by  $\alpha$.    If  $x\in \Ob(G)$, and $W\subseteq G $, we write $W_x$  for $W\cap \alpha^{-1}(x) $,  and  call   $W_x$  the {\em star} of $W$ at $x$.
Especially we write $G_x$ for $\alpha^{-1}(x)$ and call {\em star} of $G$ at $x$. We denote the set of inverses of the morphisms in $W$ by $W^{-1}$.

 A {\em star  topological groupoid} is a groupoid in which the  stars $G_x$'s have topologies such that for each $g\in G(x,y)$ the  left (and hence right) translation
 \[L_g\colon G_y\rightarrow
G_x,h\mapsto g\circ h \]
 is a homeomorphism and $G$  is the topological sum of the $G_x$'s. A {\em topological groupoid} is a groupoid in which $G$ and $\Ob(G)$ have both topologies such that the structural maps of groupoid are continuous.

 A {\em group-groupoid} is a groupoid $G$ in which  the sets of objects and morphisms have both  group structures  and the product map $ G\times G\rightarrow G, (g,h)\mapsto gh$,  inverse $G\rightarrow G, g\mapsto g^{-1}$ and, the unit maps $\colon \{\star\} \rightarrow G$,  where $\{\star\}$ is singleton,  are morphisms of groupoids.

  In a group-groupoid $G$, we write $g\circ h$ for the composition of morphisms $g$ and $h$ in groupoid while $gh$ for  the product in group and write $\bar{g}$ for the  inverse of $g$ in groupoid and $g^{-1}$ for the one in group. Here note that the product map  is a morphism of groupoids if and only if the {\em interchange rule} \[(gh)\circ (kl)=(g\circ k) (h\circ l)\] is satisfied  for $g,h,k,l\in G$ whenever one side composite is defined.

 A {\em topological group-groupoid} is defined in \cite[Definition 1]{Icen-Gursy} as a group-groupoid which is also a topological groupoid and the structural maps of group multiplication are continuous.  We define a {\em star topological group-groupoid} as a group-groupoid which is also a star topological groupoid.

Let  $X$  be a topological space admitting a simply connected cover. A subset  $U$  of  $X$  is called {\em liftable} if  $U$  is open, path-connected and the inclusion $U\rightarrow X$ maps each fundamental group of $U$  trivially. If  $U$  is liftable, and  $q\colon Y\rightarrow X$  is a covering map,  then  for  any   $y\in Y $  and  $x\in U$  such that  $qy = x $, there is a unique map   $\hat{\imath}\colon U \rightarrow Y$   such  that  $\hat{\imath} x = y$  and $q\hat{\imath}$  is the  inclusion   $U\rightarrow X $.   A space $X$ is called {\em semi-locally simply connected} if each point has a liftable neighborhood and {\em locally simply connected} if it has a base of simply connected sets.

Let  $X$  be a topological space such that each  path component of  $X$  admits a simply connected covering space.  It is  standard  that  if the fundamental groupoid  $\pi X$  is provided with a topology  as  in  \cite{Br-Da}, then for an  $x\in X$ the target point map $t \colon (\pi X)_x\rightarrow X$   is the universal covering map of $X$  based  at  $x$  (see  also Brown~\cite[10.5.8]{Br1}).

 Let  $G$ be a topological groupoid and  $W$ an open subset of $G$  including all the identities. As a generalization  of the construction in \cite{Do-La}, the   monodromy groupoid $\M(G,W)$ is defined as the quotient groupoid $\F(W)/N$, where  $\F(W)$ is the free groupoid on $W$ and $N$ is the normal subgroupoid of $\F(W)$ generated by the elements of the form $[uv]^{-1}[u][v]$ whenever $uv\in W$ for $u,v\in W$.  Then  $\M(G,W)$ has a universal property that any local  morphism $f\colon W\rightarrow H$   globalizes  to a unique morphism $\tilde{f} \colon \M(G,W) \rightarrow H$ of groupoids.

Let  $G$  be a star topological groupoid such that each star $G_x$ has a universal cover.  The  groupoid
$\Mon(G)$  is defined in \cite{Ma} as the disjoint union of the universal covers of stars $G_x$'s  at  the base points identities.  Hence $\Mon(G)$ is disjoint union of  the stars $(\pi (G_x))_{\epsilon (x)}$.  The object   set $X$ of $\Mon(G)$ is the same as that of  $G $. The source point map    $\alpha \colon \Mon(G)\rightarrow X$ maps all stars  $(\pi (G_x))_{\epsilon (x)}$ to $x $, while the target point map $\beta \colon  \Mon(G)\rightarrow X$
is defined  on each star $(\pi (G_x))_{\epsilon (x)}$ as the composition of the two target point maps
\[ (\pi (G_x))_{\epsilon (x)} \stackrel{\beta} {\longrightarrow}  G_x
\stackrel{\beta} {\longrightarrow} X .\]
 As explained in
Mackenzie~\cite[p.67]{Ma} there is a partial composition on   $\Mon(G)$ defined by
\[    [a] \bullet [b] = [a\star (a(1)\circ b) ] \] where
$\star$,  inside the bracket, denotes the usual composition of
paths and $\circ$ denotes the composition in the groupoid. Here  $a(1)\circ b$ is the path defined by  $(a(1)\circ b)(t)=a(1)\circ b(t)$ $(0\leq t\leq 1)$. Here we point that since $G$ is a star topological groupoid, the left translation is a homeomorphism. Hence the path $a(1)\circ b$, which is a left translation of $b$ by $a(1)$,  is defined when  $b$ is a path.  So the path $a\star (a(1)\circ b) $ is defined by
\[( a\star (a(1)\circ b))(t) =\left\{\begin{array}{ll}
                a(2t),  &  \mbox {$0\leq t\leq \frac{1}{2}$}\\\\
           a(1)\circ b(2t-1),   &  \mbox{$ \frac{1}{2}\leq t\leq 1$}.
                \end{array}
                \right. \]
 Here if   $a $ is a  path  in   $G_x$   from  $\epsilon (x)$
to   $a(1) $,   where $\beta( a(1)) = y $, say, and   $b$  is a path
in  $G_y$ from $\epsilon (y)$ to $b(1)$, then for  each  $t\in [0,1] $ the  composite $a(1)\circ b(t)$
is defined in  $G_y $, yielding a  path $a(1)\circ b$ from $a(1)$  to
$a(1)\circ b(1)$.  It is straightforward to prove that in this  way a  groupoid is defined on $\Mon(G)$  and that the target point map of paths induces  a  morphism  of groupoids $p \colon
\Mon(G) \rightarrow G $.

The following theorem whose Lie version is given in \cite[Theorem 4.2]{Br-Mu2},  identifies two monodromy groupoids $\M(G,W)$ and $\Mon(G)$ as star topological groupoids.

\begin{theorem}\label{isomorp}   Let $G$ be a star connected topological groupoid such that each star $G_x$ has a simply connected cover. Suppose that $W$ is a star path connected neighborhood of $\Ob(G)$  in $G$ and $W^{2}$ is contained in a star path connected neighbourhood $V$ of $\Ob(G)$ such that for all $x\in \Ob(G)$,  $V_x$ is liftable. Then there is an isomorphism of star topological groupoids $\M(G,W)\rightarrow \Mon(G)$ and hence the morphism  $\M(G,W)\rightarrow G$  is a
star universal covering map.
\end{theorem}

\section{Monodromy groupoids as star topological group-groupoids}

In  \cite[Theorem 3.10 ]{Mu-Be-Tu-Na} it was proved that if $G$ is a topological group-groupoid in which each star $G_x$ has a universal cover, then $\Mon(G)$ is a group-groupoid. We can now state the following theorem in terms of star topological group-groupoids.
\begin{theorem}\label{Mongpgpd}  Let   $G$  be a topological group-groupoid such that  each star $G_x$ has a universal cover. Then the monodromy groupoid  $\Mon (G)$ is a star topological group-groupoid.
\end{theorem}
\begin{proof} If each star  $G_x$ admits a universal  cover at $\epsilon (x)
$, then each star  $\Mon(G)_x$ may be given  a topology so that   it is
the universal cover of  $G_x$ based at $\epsilon (x) $, and then
$\Mon(G)$ becomes a star topological groupoid.  Further by the detailed proof of the Theorem  \cite[Theorem 3.10 ]{Mu-Be-Tu-Na}, we define a group structure on $\Mon(G)$   by
\[ \Mon(G)\times \Mon(G)\rightarrow \Mon(G), ([a],[b])\mapsto [ab] \]
such that  $\Mon(G)$ is a group-groupoid. The other details of the proof follow from the cited theorem.
\end{proof}

The following corollary is a result of Theorem \ref{isomorp} and  Theorem \ref{Mongpgpd}.
\begin{corollary}\label{isomongrp}
Let $G$ be a topological group-groupoid and $W$ an open subset of $G$ satisfying the conditions in Theorem \ref{isomorp}, then the monodromy groupoid $\M(G,W)$ is a star topological group-groupoid.
\end{corollary}

Let  $\TopGrpGpd$ be the category whose objects are topological group-groupoids and morphisms are the continuous groupoid morphisms preserving group operation; and let  $\TGrpGpd$ be the full subcategory of  $\TopGrpGpd$ on those objects which are topological group-groupoids whose stars have universal covers.  Let  $\STopGrpGpd$ be the category whose objects are  star topological group-groupoids and the morphisms are those of group-groupoids which are continuous on stars. Then we have the following.
\begin{proposition} \[\Mon\colon \TGrpGpd\rightarrow \STopGrpGpd  \]
which assigns  the monodromy groupoid $\Mon(G)$ to such a topological group-groupoid $G$ is a functor.
\end{proposition}
\begin{proof} We know from Theorem \ref{Mongpgpd} that if $G$ is a topological group-groupoid in which the stars have universal covers, then  $\Mon(G)$ is also a star topological group-groupoid. Let  $f\colon G\rightarrow H$ be a morphism of  $\TGrpGpd$.  Then   the restriction $f\colon G_x\rightarrow H_{f(x)}$ is continuous and hence by \cite[Proposition 3]{Br-Da}, the induced  morphism $\pi(f)\colon \pi(G_x)\rightarrow \pi(H_{f(x)})$, which is a morphism of topological groupoids, is continuous. Latter morphism is restricted to the continuous map
$\pi(f)\colon \pi(G_x)_{1_x}\rightarrow \pi(H_{f(x)})_{1_{f(x)}}$ which is $\Mon(f)\colon (\Mon(G))_x\rightarrow (\Mon(H))_{f(x)}$. That means  $\Mon(f)$ is a morphism of star topological group-groupoids.  The other details of the proof is straightforward.
\end{proof}

We need the following results  in the proof of Theorem \ref{MonProduct}.

\begin{proposition}  {\rm \cite[Theorem 1]{Br-Da}}  \label{Fundgdtopgd} If $X$ is a locally path connected and
semi-locally simply connected space, then the fundamental groupoid $\pi X$ may be given a topology making it  a topological groupoid.
\end{proposition}

\begin{theorem} \label{Productgrpd}{\rm \cite[Theorem 3.8]{Mu-Be-Tu-Na}}\label{Theocarptopgpd} If    $X$ and $Y$ are  locally path connected and
locally simply connected topological spaces, then  $\pi (X\times
Y)$ and $\pi (X)\times \pi (Y)$ are isomorphic as topological
groupoids.
\end{theorem}

\begin{theorem}\label{MonProduct} For the topological group-groupoids $G$ and $H$ whose stars have universal covers, the monodromy groupoids  $\Mon(G\times H)$ and $\Mon(G)\times \Mon(H)$ as star topological group-groupoids are isomorphic.
\end{theorem}
\begin{proof} Let $G$ and $H$ be the topological group-groupoids such that the stars have the universal covers. Then by Theorem \ref{Mongpgpd},  $\Mon(G\times H)$ and $\Mon(G)\times \Mon(H)$ are star topological group-groupoids and by  \cite[Theorem 2.1]{Mu-Be-Tu-Na} we know that these groupoids $\Mon(G\times H)$ and $\Mon(G)\times \Mon(H)$ are isomorphic. By the fact that the star $(\Mon(G))_x$ is the star $(\pi(G_x))_{\epsilon(x)}$   of the fundamental groupoid $\pi(G_x)$ we have the following evaluation

\begin{align*}
  (\Mon(G\times H))_{(x,y)}&= \pi((G\times H)_{(x,y)})_{\epsilon(x,y)}\\
  &=(\pi(G_x\times H_y))_{\epsilon(x,y)}\\
   &\simeq (\pi(G_x))_{\epsilon(x)}\times (\pi(H_y))_{\epsilon(y)}\tag{by Theorem  \ref{Productgrpd}} \\
   &=(\Mon(G))_x\times (\Mon(H))_y
\end{align*}

Hence we have a homeomorphism $(\Mon(G\times H))_{(x,y)}\rightarrow (\Mon(G))_x\times (\Mon(H))_y$ on the stars and then by gluing these homeomorphisms on the stars we have an isomorphism $f\colon \Mon(G\times H)\rightarrow \Mon(G)\times \Mon(H)$ defined by $f([a])=([p_1 a],[p_2 a])$ for $[a]\in \Mon(G\times H)$, which is identity on objects.  Here $f$ is reduced to the homeomorphisms on the stars  and it  is also a morphism of group-groupoids. Hence $f$ is an isomorphism of star topological group-groupoids and therefore the star topological group-groupoids  $\Mon(G\times H)$ and $\Mon(G)\times \Mon(H)$  are isomorphic.
\end{proof}

Before giving the group-groupoid version of the monodromy principle we give definition of local morphism for group-groupoids adapted from definition of local morphism of groupoids in \cite{Mu1}.

\begin{definition} \rm
Let $G$ and $H$ be group-groupoids. A {\em local morphism}  from $G$ to $H$ is a map $f \colon W\rightarrow H$ from a subset $W$ of $G$, including the identities, satisfying the  conditions  $\alpha_H(f(u))=f(\alpha_G (u))$, $\beta_H(f(u))=f(\beta_G (u))$,  $f(u\circ v)=f(u)\circ f(v)$ and $f(uv)=f(u)f(v)$  whenever  $u, v \in W$, $u\circ v\in W$ and $uv\in W$.

A {\em local morphism of star topological  group-groupoids} is a local morphism of group-groupoids which is continuous on the stars.

Let  $G$ and $H$ be topological group-groupoids and $W$  an open neighborhood of $\Ob(G)$.  A {\em local morphism} from $G$ to $H$  is a continuous local morphism  $f\colon W\rightarrow H$  of group-groupoids.
\end{definition}

We  can now prove a weak monodromy principle for star topological group-groupoids.

\begin{theorem}\label{weak} {\em  (Weak Monodromy Principle) }
Let  $G$   be a star connected topological group-groupoid and  $W$ an open and star connected subgroup of  $G$ containing  $\Ob(G)$ and satisfying the condition in Theorem \ref{isomorp}. Let  $H$  be a star topological group-groupoid  and  $f \colon W\rightarrow H$  a local morphism  of star topological group-groupoids which is  identity on $\Ob(G)$.  Then  $f$ globalizes  uniquely   to a morphism  $\tilde{f} \colon \M(G,W)\rightarrow H$ of star topological group-groupoids.
\end{theorem}
\begin{proof} Here  remark  that by  Corollary \ref{isomongrp},  $\M(G,W)$ is a star topological group-groupoid and by
the construction of $\M(G,W)$  we have an inclusion map $\imath \colon W\rightarrow \M(G,W)$ and  $W$ is homeomorphic to  $\imath(W)=W'$ which generates $\M(G,W)$.   The existence of $\tilde{f}\colon \M(G,W)\rightarrow H$  as  a groupoid morphism follows from the universal property of free groupoid $\F(W)$ and  the fact that $\M(G,W)$ is generated by $W'$.  Hence one needs to show that $\tilde{f}$ is a group-groupoid morphism, i.e., it preserves the group  operation.  Let  $a$ and $b$ be the morphisms of $\M(G,W)$. Since $W'$ generates $\M(G,W)$, $a$ and $b$ are written as $a= u_1\circ u_2 \ldots \circ  u_n$ and $b= v_1 \circ v_2 \ldots\circ v_m $ for $u_i,v_j\in W'$.  Since $W$ is group and so also is $W'$ we have $u_i v_i \in W'$ for all $i$. Then by the interchange rule we  have the following evaluation for $m\geq n$
\begin{align*}
  \tilde{f}(a  b)&= \tilde{f}((u_1\circ u_2 \ldots \circ  u_n) (v_1 \circ v_2 \ldots\circ v_m))\\
  &=\tilde{f}( u_1 v_1\circ \ldots \circ u_n  v_n\circ 1_{\beta(u_n)}v_{n+1}\circ\ldots \circ 1_{\beta(u_n)}v_m) \\
   &=f( u_1 v_1\circ \ldots \circ u_n  v_n\circ 1_{\beta(u_n)}v_{n+1}\circ\ldots \circ 1_{\beta(u_n)}v_m)  \\
   &=f(u_1) f(v_1)\circ \ldots \circ f(1_{\beta(u_n)})f(v_m)\\
   &=(f(u_1)\circ \ldots\circ f(u_n))  (f(v_1)\circ \ldots \circ f(v_m)) \\
    &=f(u_1\circ \ldots\circ u_n)  f(v_1\circ \ldots \circ v_m) \\
   &=\tilde{f}(a)  \tilde{f}(b).
   \end{align*}
Since $\M(G,W)$ is generated by $W'$, the continuity of $\tilde{f}\colon \M(G,W)\rightarrow H$ on stars follows by the continuity of $f\colon W\rightarrow H$.
\end{proof}

As a result of Theorem \ref{weak} we have the following corollary.
\begin{corollary} \label{weakmonodromyprinciple}
Let  $G$  be a star topological group-groupoid which is star connected and star simply connected and let $W$ be an open and star connected subgroup of  $G$ containing  $\Ob(G)$ and satisfying the condition in Theorem \ref{isomorp}.   Let  $H$  be a star topological group-groupoid  and let $f \colon W\rightarrow H$  be a local morphism  of star topological group-groupoids which is  identity on $\Ob(G)$.  Then  $f$   globalizes uniquely to a morphism  $\tilde{f} \colon G\rightarrow H$ of star topological group-groupoids.
\end{corollary}
\begin{proof} Since $G$ is star connected and star simply connected,  $\Mon(G)$, as a star topological groupoid, becomes isomorphic  to $G$;  and by Theorem \ref{isomorp} and  Corollary
  \ref{isomongrp} $\M(G,W)$ and $\Mon(G)$ are isomorphic as star topological groupoids. Hence the rest of the proof follows from Theorem  \ref{weak}.\end{proof}

\section{Topological  structure on monodromy groupoid as group-groupoid}
In this section we prove that if $G$ is a topological group-groupoid in which each star has a universal cover  and $W$ is a useful open subset of $G$, including the identities, then the monodromy groupoid $\M(G,W)$ becomes a topological group-groupoid with the topology obtained by Theorem \ref{TheoremHolonomy}.

Let $G$ be a groupoid and $X=\Ob(G)$  a topological space. An \textit{admissible local section} of $G$, which is due to Ehresmann  \cite{Eh},  is a function $\sigma\colon U\rightarrow G$ from an open subset of $X$ such that the following holds.
\begin{enumerate}
\item $\alpha\sigma(x)=x$ for all $x\in U$;
\item $\beta\sigma(U)$ is open in $X$;
\item $\beta\sigma$ maps $U$ topologically to $\beta\sigma(U)$.
\end{enumerate}
Here the set $U$ is called {\em domain} of $s$ and written as $D_s$. Let $\Gamma(G)$ be the set of all admissible local sections of $G$. A product defined on $\Gamma(G)$ as follows: for any two admissible local sections
\[(st)x=(sx)(t\beta sx) \]
$s$ and $t$ are composable if $D_{st}=D_s$. If $s$ is admissible local section then $s^{-1}$ is also an admissible local section $\beta sD_s\rightarrow G$, $\beta sx \mapsto (sx)^{-1}$.

Let $W$ be a subset of $G$ and let $W$ have a topology such that $X$ is a subspace. $(\alpha,\beta,W)$ is called \textit{enough continuous admissible local sections} or   {\em locally sectionable} if
\begin{enumerate}
\item $s\alpha(w)=w$;
\item $s(U)\subseteq W$;
\item $s$ is continuous from $D_s$ to $W$. Such $s$ is called continuous admissible local section.
\end{enumerate}

Holonomy groupoid  is constructed for a locally topological groupoid whose definition is as follows (see \cite{Br-Mu3} for a locally topological groupoid structure on a foliated manifold):

\begin{definition}\label{ltopgrpd} \rm \cite[Definition 1.2]{Ao-Br}
A locally topological groupoid is a pair $(G,W)$ where $G$ is a groupoid and $W$ is a topological space such that
\begin{enumerate}
\item $\Ob(G) \subseteq W\subseteq G$;
\item $W=W^{-1}$;
\item $W$ generates $G$ as a groupoid;
\item The set $W_\delta =(W\times_\alpha W)\cap \delta^{-1}(W)$ is open in $W\times_\alpha W$ and the restriction to $W_\delta$ of the difference map
$\delta\colon G\times_\alpha G \rightarrow G$ is continuous;
\item the restrictions to $W$ of the source and target point maps $\alpha$, $\beta$ are continuous and $(\alpha,\beta,W) $ has enough continuous admissible local sections.
\end{enumerate}
\end{definition}
Note that a topological groupoid is a locally topological groupoid but converse is not true.

The following globalization theorem assigns a topological groupoid called {\em holonomy groupoid} and denoted by $\Hol(G,W)$ or only $H$  to a locally topological groupoid $(G,W)$ and hence it is more useful to obtain an appropriate topology on the monodromy groupoid. We give an outline  of the proof since some details of the construction in the proof   are needed for Proposition \ref{Extendibility} and  Theorem \ref{isohol}.

\begin{theorem} \label{TheoremHolonomy} {\rm \cite[Theorem 2.1]{Ao-Br}} Let $(G,W)$  be a locally topological groupoid. Then there is a
topological  groupoid $H$, a morphism $\phi : H \rightarrow
G$  of groupoids and an embedding $i : W \rightarrow H$
 of $W$  to an open neighborhood of $\Ob(H)$  such
that the following  conditions are satisfied.

i) $\phi $ {\it is the identity on objects}, $\phi i =
id_{W} ,  \phi ^{-1}(W)$ {\it is open in} $H$,and the
restriction  $\phi _{W} : \phi^{-1}(W)  \rightarrow W$ {\it of}
$\phi $ {\it is continuous};

 {\it ii) if A is a} topological {\it groupoid and} $\xi :  A
\rightarrow G$ {\it is a morphism of groupoids such that:

a) $\xi $ is the identity on objects;

b) the restriction} $\xi _{W} : \xi ^{-1}(W) \rightarrow W$ {\it
of}  $\xi $ {\it is continuous and} $\xi ^{-1}(W)$ {\it is open in A
and generates A;}

$c)$ {\it the triple} $(\alpha _{A} , \beta _{A} , A)$ {\it is
locally  sectionable,}

\noindent {\it then there is a unique morphism} $\xi ^\prime : A
\rightarrow H$  {\it of topological  groupoids such that} $\phi \xi
^\prime = \xi $  {\it and} $\xi ^\prime a = i\xi a$ {\it for} $a
\in \xi ^{-1}(W)${\it .}
\end{theorem}


\begin{proof}
Let $\Gamma (G)$ be the set of all admissible local sections of
$G $.  Define a product on $\Gamma (G)$ by \[(st)x =(sx) (t\beta
sx)\] for two  admissible local sections $s$ and $t$. If $s$
is an admissible local  section then write $s^{-1}$ for the
admissible local section $\beta  s D_s \rightarrow G,
\beta sx \mapstochar \rightarrow (sx)^{-1}$.  With this product
$\Gamma (G)$ becomes an inverse semigroup. Let  $\Gamma ^{c}(W)$
be the subset of $\Gamma (G)$ consisting of admissible  local
sections which have values in $W$ and are continuous. Let  $\Gamma
^{c}(G, W)$ be the subsemigroup of $\Gamma (G)$ generated by
$\Gamma ^{c}(W)$. Then $\Gamma ^{c}(G, W)$ is again an inverse
semigroup.  Intuitively, it contains information on the
iteration of local procedures.  \par

Let $J(G)$ be the sheaf of germs of admissible local sections of
$G$.  Thus the elements of $J(G)$ are the equivalence classes of
pairs $(x,s)$  such that $s \in \Gamma (G), x \in { D_s}$,
and $(x,s)$ is equivalent to $(y,t)$ if and only if $x = y$ and
$s$ and $t$ agree on a neighbourhood of  $x$. The equivalence
class of $(x,s)$ is written $[s]_{x}$. The product  structure on
$\Gamma (G)$ induces a groupoid structure on $J(G)$ with $X$  as
the set of objects, and source and target point maps are $[s]_{x}
\mapstochar  \rightarrow x, [s]_{x} \mapstochar \rightarrow
\beta sx$ respectively. Let $J^{c}(G, W)$  be the subsheaf of $J(G)$ of germs
of elements of $\Gamma ^{c} (G,W)$.  Then $J^{c} (G, W)$ is
generated as a subgroupoid of $J(G)$ by the sheaf  $J^{c} (W)$
of germs of elements of $\Gamma ^{c} (W)$. Thus an element of
$J^{c} (G, W)$ is of the form \[[s]_{x} = [s_1]_{x_1} \ldots
[s_n]_{x{_n}}\]        where $s = s_1 \ldots s_n$ with
$[s_i]_{x_{i}} \in J^{c}(W), x_{i+1}  = \beta s_{i}x_{i} , i =
1,\ldots ,n$ and $x_1 = x \in { D_s}$. \par

Let $\psi : J(G) \rightarrow G$ be the map defined by
$\psi ([s]_x) = s(x) $, where $s$ is an admissible local
section.  Then $\psi (J^{c} (G, W)) = G $. Let $J_0 = J^{c} (W)
\cap \ker \psi $.  Then $J_0$ is a normal subgroupoid of $J^{c}
(G, W) $; the proof is  the same as in \cite[Lemma 2.2]{Ao-Br}
The holonomy groupoid  $H =  \Hol(G, W)$ is defined to be the
quotient groupoid $J^{c} (G, W)/J_0$.  Let $p\colon J^{c}(G, W) \rightarrow H$
be the quotient morphism and  let $p([s]_{x})$ be denoted by
$<s>_{x}$.  Since $J_{0} \subseteq \ker \psi $ there is a
surjective morphism  $\phi : H \rightarrow G$ such that $\phi p
= \psi $.\par

The topology on the holonomy groupoid $H$ such that $H$ with
this topology  is a topological groupoid is constructed as follows.  Let
$s \in \Gamma ^{c}(G, W)$. A partial function $\sigma _s :  W
\rightarrow H$ is defined as follows. The domain of $\sigma _s$
is  the set of $w \in W$ such that $\beta w \in {D_s}$. A
continuous  admissible local section $f$ through $w$ is chosen and
the value $\sigma _{s}w$ is defined to be \[\sigma _{s}w =<f>_{\alpha w}<s>_{\beta w} = <fs>_{\alpha w}.\] It is proven
that $\sigma _{s}w$ is independent of the choice of the local
section $f$ and that these $\sigma _{s}$ form a set of charts.
Then the initial topology with respect to the charts $\sigma
_{s}$ is imposed on  $H$. With this topology $H$ becomes a topological groupoid.  Again the proof is essentially the same as in
Aof-Brown~\cite{Ao-Br}.
\end{proof}

 From the construction of the holonomy groupoid the  following extendibility condition is obtained.

\begin{proposition} \label{Extendibility} {\it The locally} topological {\it groupoid}  $(G,W)$ {\it
is extendible to} a topological {\it groupoid structure on}  $G$ {\it if
and only if the following condition holds:}  \par \noindent
{\bf (1)}: {\it if} $x \in \Ob(G)$ {\it , and} $s$ {\it is a
product} $s_{1}  \ldots s_n$ {\it of local sections about} $x$
{\it such that each} $s_i$  {\it lies in} $\Gamma ^{c}(W)$ {\it
and} $s(x) = 1_x${\it , then there is a  restriction} $s^\prime
$ {\it of} $s$ {\it to a neighbourhood of}  $x$ {\it such that}
$s^\prime $ {\it has image in} $W$ {\it and is} continuous, {\it
i.e.} $s^\prime \in \Gamma ^{c}(W)$.\label{locex}
\end{proposition}

To prove that $\M(G,W)$ is a topological group-groupoid, we first prove a  more general result.
\begin{theorem}\label{isohol}
Let $G$ be a topological group-groupoid and $W $  an open subset of $G$ such that
\begin{enumerate}
\item $\Ob(G) \subseteq W$
\item $W=W^{-1}$
\item $W$ generates $G$ and
\item $(\alpha_W,\beta_W,W)$ has enough continuous admissible local sections.
\end{enumerate}
Let $p\colon M\rightarrow G$ be a morphism of group-groupoids such that $\Ob(p)\colon \Ob(M)\rightarrow \Ob(G)$ is identity and assume that  $\imath\colon W\rightarrow M$ is an inclusion   such that $p\imath=i\colon W\rightarrow G$ and $W'=\imath(W)$ generates $M$.

Then $M$ admits the structure of a topological group-groupoid such that $p\colon M\rightarrow G$ is a morphism of topological group-groupoids and maps $W'$ to $W$ homeomorphically.
\end{theorem}

\begin{proof}  As it was proved in \cite[Corollary 5.6]{Br-Mu2}, $(M,W')$ is a locally topological groupoid and by Proposition \ref{Extendibility} it is extendible, i.e., the holonomy groupoid  $H=\Hol(M,W')$ is isomorphic to $M$. Hence by Theorem \ref{TheoremHolonomy},  $M$ becomes a topological groupoid  such that  $M$ has the  chart topology from  $W'$. Hence the chart open subsets of $M$ form a base for this topology.  We now prove that the difference map of product   $m\colon M\times M\rightarrow M, (a,b)\rightarrow ab^{-1}$  is continuous. We now consider the following diagram.

\begin{center}
\begin{tikzpicture}
\matrix(a)[matrix of math nodes, row sep=3em, column sep=3.5em, text height=1.5ex, text depth=0.25ex]
{ m^{-1}(W')\cap (W'\times W')&W'\\ M\times M&M\\};
\path[->,font=\scriptsize](a-1-1) edge node[above]{$m_{W'}$} (a-1-2);
\path[->,font=\scriptsize](a-1-1) edge node[left]{$$} (a-2-1);
\path[->,font=\scriptsize](a-1-2) edge node[right]{$$} (a-2-2);
\path[->,font=\scriptsize](a-2-1) edge node[above]{$m$} (a-2-2);
\end{tikzpicture}
\end{center}

Let $U$ be a base open subset, i.e., a chart open subset of $M$ and $U'$ be the open subset of $W'$, which is homeomorphic to $U$. Since the restriction
$m_{W'}\colon m^{-1}(W')\cap (W'\times W')\rightarrow W'$ is continuous,  the inverse image $(m_{W'})^{-1}(U')$ is open in  $m^{-1}(W')\cap (W'\times W')$ and it is homeomorphic  to an  open neighbourhood of $M\times M$. That means the inverse image $m^{-1}(U)$ is open in $M\times M$ and hence $m$ is continuous.

Since the locally topological groupoid $(M,W')$ is extendible the holonomy groupoid  $H=\Hol(M,W')$  is isomorphic   to $M$ and hence by Theorem \ref{TheoremHolonomy},  $p\colon M\rightarrow G$ becomes a morphism of topological groupoids. Further by assumption it is a morphism of group-groupoids. Hence $p\colon M\rightarrow G$ becomes a morphism of topological group-groupoids.
\end{proof}

As a result of Theorems \ref{isomorp} and \ref{isohol} we can state the following Corollary.
\begin{corollary}\label{Mongpgpdintgpd}
Let   $G$  be  a  topological group-groupoid such that  each star $G_x$ has a universal cover. Suppose that $W$ is a star path connected neighborhood of  $\Ob(G)$  in $G$ satisfying the conditions in  Theorem \ref{isomorp} and Theorem \ref{isohol}. Then the monodromy groupoid  $\Mon (G)$ is a topological group-groupoid such that the projection $p\colon \Mon(G)\rightarrow G$ is a morphism of topological group-groupoids.
\end{corollary}
\begin{proof} By  Theorem \ref{isomorp}, $\M(G,W)$ and $\Mon(G)$ are isomorphic as star topological groupoids. By Theorem \ref{isohol},  $\M(G,W)$ is a topological group-groupoid and so also is  $\Mon(G)$ as required.
\end{proof}

As a result of Corollaries \ref{weakmonodromyprinciple} and \ref{Mongpgpdintgpd} we can give the following theorem which we call as  {\em strong monodromy principle} for topological group-groupoids.
\begin{theorem}\label{strong} {\em  ( Strong Monodromy Principle) }
Let  $G$   be a star connected and star simply connected topological group-groupoid and let $W$ be an open and star connected subgroup of  $G$ satisfying the conditions of Theorem \ref{isomorp}  and  Theorem \ref{isohol}. Let  $H$  be a topological group-groupoid  and let $f \colon W\rightarrow H$  be a local morphism  of topological group-groupoids which is the identity on $\Ob(G)$. Then  $f$   extends uniquely to a morphism  $\tilde{f} \colon G\rightarrow H $ of topological group-groupoids.
\end{theorem}
\begin{proof}

By Corollary \ref{weakmonodromyprinciple}, the local morphism $f \colon W\rightarrow H$ extends to $\tilde{f}\colon G\rightarrow H$; and  by Corollary\ref{Mongpgpdintgpd} $\Mon(G)$ and $\M(G,W)$ are isomorphic as topological group-groupoids.  The continuity of $\tilde{f}$ follows from the fact that $\tilde{f}$ is continuous on an open subset  $W'$ which generates $\M(G,W)$. \end{proof}

\section*{Acknowledgment} We would like to thank the referee for bringing the paper \cite{Icen-Gursy} to our attention for topological group-groupoids

\end{document}